\documentclass{article}

\usepackage{amsfonts,amsmath, amsthm, amscd, amssymb, graphicx, color, mathrsfs, stmaryrd}

\DeclareMathOperator{\Ext}{Ext} 
\DeclareMathOperator{\diam}{diam} 

\begin{document}

\newtheorem{thm}{Theorem}[section]
\newtheorem{theo}[thm]{Theorem}
\newtheorem{prop}[thm]{Proposition}
\newtheorem{coro}[thm]{Corollary}
\newtheorem{lema}[thm]{Lemma}
\newtheorem{defi}[thm]{Definition}
\newtheorem{ejem}[thm]{Example}
\newtheorem{rema}[thm]{Remark}
\newtheorem{fact}[thm]{Fact}
\newtheorem{open}[thm]{PROBLEM}

\newcommand{\lip}{{\rm Lip}}
\newcommand{\lipo}{ {\rm Lip}_0 }
\newcommand{\conv}{\operatorname{conv}}
\newcommand{\cconv}{\overline{\conv}}

\title{Representation in $C(K)$ by Lipschitz functions}
\author{M. Raja\thanks{This research has been supported by:  Fundaci\'on S\'eneca -- ACyT Regi\'on de Murcia, project 21955/PI/22; and grant PID2021-122126NB-C32  funded by MCIN/AEI/ 10.13039/501100011033 and by “ERDF A way of making Europe”, by the EU.}}
\date{June, 2024}

\maketitle

\begin{abstract}
The isometric universality of the spaces $C(K)$ for $K$ a non scattered Hausdorff compact does not take into account the ``quality'' of the representation. Indeed, the existence of an isometric copy of a separable Banach space $X$ into $C(K)$ made of regular enough functions, say Lipschitz with respect to a lower semicontinuous metric defined on $K$, imposes severe restrictions to both $X$ and $K$. In this paper, we present a systematic treatment of the representation of Banach spaces into $C(K)$ by Lipschitz functions improving previous results of the author.
\end{abstract}




\section{Introduction}

A celebrated result of Banach and Mazur says that any separable Banach space $X$ can be found linearly isometrically embedded into $C[0,1]$, the space of continuous functions over the unit interval endowed with the supremum norm. A different matter is to find explicit representations simple Banach spaces. For instance, 
$$ (x,y) \rightarrow  x \cos(\pi t) + y \sin(\pi t) , ~~t \in [0,1]$$
gives a representation of the Euclidean plane $({\Bbb R}^2, \| \cdot \|_2)$ into $C[0,1]$ by fairly nice functions. However, it is not easy at all to the same with $({\Bbb R}^3, \| \cdot \|_2)$. Indeed, the functions witnessing the isometric embedding cannot be even Lipschitz. The reason, first noticed by W. F. Donoghue \cite{donoghue}, is related to gap between the topological dimensions of the interval $[0,1]$ and of the sphere ${\Bbb S}^2= \partial B_{{\Bbb R}^3}$. We may consider a more general problem: given a Hausdorff compact space $K$ and a finer metric $d$, study the Banach spaces that isometrically embed into $C(K)$ as a subset of $d$-Lipschitz functions.\\

Some results about this topic appeared scattered in several papers \cite{raja1, raja2, raja3, jonard_raja}. The aim of this note is to bring together the main ideas on {\it Lipschitz subspaces of} $C(K)$ with improvements and some new results. Despite the quite mathematical insignificance of our original problem, we have found very interesting the connections between Topology and Banach space theory motivated by this research. Moreover, there are some ties with other trending topics, such as {\it lineability} (see the monograph \cite{linea})  and {\it Lipschitz-free Banach spaces} (see the book \cite{Wea}, although the author prefers the name {\it Arens-Eells} for these spaces).\\

The paper is organised as follows. The next section gathers some tools and generalities on compact spaces endowed with a lower semicontinuous metric. The third section deals with properties of Lipschitz subspaces both from the isomorphic and isometric points of view, however the results in the isometric theory are much more satisfactory. The fourth section covers the case when the metric actually metrizes $K$. The fifth section is devoted to the use of fragmentability, the main property making a difference for strictly finer lower semicontinuous metrics. So far in our paper, the only infinite-dimensional Lipschitz subspaces identified into $C(K)$  are copies of $c_0$ or $\ell_1$, thus the last section deals with Lipschitz embeddings and obstructions for other Banach spaces. In order to make this note more independent of our previous work on the subject, we have included some sketched proofs of the quoted key results, or complete proofs in case they are simpler here or lead to a statement more general than in the original paper.\\

All the Banach spaces considered are real and all the compact spaces are supposed to be Hausdorff. Our notation is totally standard and we address to generic references for any unexplained definition \cite{JL, banach}.

\section{The bitopological setting}

In Banach space theory, it is usual to deal with more than one topology: the norm and weak topologies, the weak$^*$ in case of dual Banach spaces, or the pointwise for function spaces. That bitopological setting, actually one topology and a metric generating a finer topology, can be discussed more generally. We will assume compactness for the coarse topology, so we will mainly use $K$ to denote the space, and  lower semicontinuity (lsc) for the metric as a function $d:K \times K \rightarrow [0,+\infty)$.

\begin{theo}
Let $K$ be a compact space and $d$ be a lower semicontinuous metric defined on $K$. Then:
\begin{itemize}
\item[(a)] the topology generated on $K$ by $d$ is finer;
\item[(b)] $K$ is complete when endowed with $d$;
\item[(c)] for every closed subset $H \subset K$ and $r>0$, the set 
$$ B[H,r] = \{x \in K: d(H,x) \leq r\} $$
is closed;
\item[(d)] every $f \in C(K)$ is $d$-uniformly continuous.
\end{itemize}
\end{theo}

\begin{proof}
Statements (a) and (b) are included in \cite[Lemma 2.1]{JNR}, meanwhile  (c) is \cite[Lemma 2.2]{JNR}. We left (d), which is an easy exercise, to the reader.
\end{proof}

From now on $K$ will be a compact Hausdorff space together with a lower semicontinuous metric denoted $d$. The supremum norm of $C(K)$ is denoted $\| \cdot \|$, as any generic norm, and the notation $\| \cdot \|_\infty$ is used  to avoid confusions when necessary.
In case $K$ is a weak$^*$ compact of a dual Banach space (notably, the dual unit ball), the metric $d$ will be the one induced by the norm, unless otherwise stated.\\

The Lipschitz constant of a real function $f:K \rightarrow {\Bbb R}$ is defined as follows 
$$ L(f) = \sup \left\{ \frac{|f(t_1)-f(t_2)|}{d(t_1,t_2)}: t_1,t_2 \in K, t_1 \not = t_2 \right\} .$$
A real function $f$ defined on $K$ is said to be Lipschitz if $L(f) < +\infty$.
The set of Lipschitz functions defined on $K$ will be denoted $\lip(K)$. We will also consider Lipschitz mappings between metric spaces, for whom the Lipschitz constant is defined likewise.\\

The following important result due to Eva Kopecká \cite{Eva} implies, among other things, the great availability of functions that are both continuous and $d$-Lipschitz.

\begin{theo}\label{interpol}
Let $H \subset K$ be closed and let $f: H \rightarrow [a,b]$ be continuous and Lipschitz. Then there exists $\tilde{f}: K \rightarrow [a,b]$ being continuous and Lipschitz with $L(\tilde{f})=L(f)$ such that $\tilde{f}|_H=f$.
\end{theo}

Now we will see two ways to linearize compacta with lower semicontinuous metrics. Fix a point $t_0 \in K$ and let 
$$ \lipo(K)= \{f \in \lip(K): f(t_0)=0\}. $$

Using Theorem \ref{interpol}, we can give a simpler proof of the following result of Jayne, Namioka and Rogers  \cite[Theorem 2.1]{JNR}.

\begin{theo}
Let $K$ be a compact space and $d$ be a bounded lower semicontinuous metric. Then there exists a Banach space $X$ such that $K$ embeds as a $w^*$-compact subset of $X^*$ and $d$ coincides with the metric induced by the norm of $X^*$.
\end{theo}

\begin{proof} Take $X = \lipo(K) \cap C(K)$ with the Lipschitz seminorm $L$, that is an actual norm here. 
Note that $X$ is a Banach space because the convergence with respect to the Lipschitz seminorm induces the uniform convergence on bounded sets of  $\lipo(K)$.
Denote by $\hat{t}(f) = f(t)$ the evaluation. Obviously, the assignment $t \rightarrow \hat{t}$ is an homeomorphism. 
Note that for any $t_1, t_2 \in K$ with $t_1 \not = t_2$ there is $f \in C(K) \cap \lip(K)$ with $L(f)=1$ such that $f(t_2) -f(t_1) =d(t_1,t_2)$. Indeed, apply Theorem \ref{interpol} to $f$ defined on $\{t_1,t_2\}$ by $f(t_1)=0, f(t_1)=d(t_1, t_2)$. Adding a constant we may even get that $f(t_0)=0$ and thus $f \in B_X$. Now, for any two points $t_1,t_2 \in K$, we have
$$ d(t_1, t_2) = \sup \{ |f(t_1) - f(t_2)|: f \in B_X \}$$
$$ = \sup \{ | \hat{t}_1(f) - \hat{t}_2(f)  |: f \in B_X \} = \| \hat{t}_1 - \hat{t}_2 \| ,$$
as wished.
\end{proof}

The previous construction allows the linearization of mappings in the following sense: a continuous mapping between compacts that is also Lipschitz extends to a linear weak$^*$ continuous mapping, obviously Lipschitz for the norms. Let us mention the relation to the Lipschitz free spaces (Arens-Eells spaces in \cite{Wea}). Given a metric space $M$ with distinguished point $0 \in M$ (a {\it pointed metric space}), we consider as before $ \lipo(M)= \{f \in \lip(M): f(t_0)=0\}$. On  $ \lipo(M)$ the Lipschitz seminorm $L$ is an actual norm. It is possible to prove that $(\lipo(M), L)$ is isometric to a dual Banach space. The Lipschitz free space generated by $M$ is a predual for $\lipo(M)$ that can be identified by the completed linear span of $M$ into 
$\lipo(M)^*$, where the point $0$ becomes the origin of the vector space. 
The Lipschitz free spaces have been intensely studied over the last two decades, however we can not benefit from the research as an additional topology on $M$ does not play a role, as is our case.\\

We finish this section with an interesting result of Benyamini \cite{Benya}.

\begin{theo}
Let $K$ be a metrizable compact space. There is weak$^*$ continuous retraction from $C(K)^*$ onto $B_{C(K)^*}$ that is also Lipschitz with constant one.
\end{theo}


\section{General results}

As said before, the compact Hausdorff space $K$ is given together with a lower semicontinuous metric $d$. Let us stress that $C(K)$ refers to the continuous real functions with respect to the compact topology, meanwhile $\lip(K)$ stands for the Lipschitz real functions with respect to the metric $d$. In general, those sets are not contained in one another, however the intersection is rich enough to recover either the topology or the metric. The following is our main definition.

\begin{defi}
A subspace $X \subset C(K)$ is said to be Lipschitz if all the elements of $X$ are Lipschitz functions, that is, if $X \subset \lip(K)$ as a subset.
\end{defi}

Even though the definition is quite clear, we will illustrate it with an example.

\begin{ejem}
Let $K=B_{\ell_2}$ be endowed with the weak topology of $\ell_2$ and let $d$ be the norm (Hilbert) metric. Consider the functions $f_n:K \rightarrow [0,1]$ defined by
$$ f_n((x_k)_{k \in {\Bbb N}}) = x_n^2. $$
Then the sequence $(f_n)$ spans a Lipschitz subspace of $C(K)$ isometric to $c_0$.
\end{ejem}

\begin{proof}
We will perform the computations, although the particular existence of isometric Lispchitz embeddings of $c_0$ into $C(K)$ can be deduced from Theorem \ref{iso_char} or Theorem \ref{cecero}.
Note that for any bounded sequence $(a_k)$ we have
$$ \sup\{ |a_k|: k \in {\Bbb N}\} \leq \sup \{ \sum_{k=1}^\infty a_k x_k^2: (x_k) \in K\} , $$
but 
$$ \sum_{k=1}^\infty a_k x_k^2 \leq \sup\{ |a_k|: k \in {\Bbb N}\} \, \sum_{k=1}^\infty x_k^2 \leq \sup\{ |a_k|: k \in {\Bbb N}\} $$
for $(x_k) \in K$. Therefore
$ \| \sum_{k=1}^\infty a_k f_k \|  = \| (a_k)\|_\infty .$
If $(a_k) \in c_0$ we get also the uniform convergence of the series, so it defines an element of $C(K)$.
As to the Lipschitzness, let $(x_k), (y_k) \in K$. Then
$$ \left| \sum_{k=1}^\infty a_k x_k^2 - \sum_{k=1}^\infty a_k y_k^2 \right| \leq \sum_{k=1}^\infty |a_k | \, |x_k+y_k| |x_k -y_k|$$
$$ \leq  \left( \sum_{k=1}^\infty |a_k |^2 |x_k+y_k|^2   \right)^{1/2}  \left(  \sum_{k=1}^\infty |x_k-y_k|^2 \right)^{1/2} \leq  2 \| (a_k)\|_\infty \, d( (x_k), (y_k))$$
as wished.
\end{proof}

The following simple application of Baire's theorem was written in similar terms in \cite{jonard_raja}, although it was only applied in the case $d$ metrizes $K$.

\begin{prop}\label{easy}
Let $X \subset C(K)$ be a non trivial linear subspace. Then either
\begin{itemize}
\item[(a)] $X \cap \lip(K)$ is of first category in $X$;
\item[(b)] or $X \subset \lip(K)$, that is, $X$ is a Lipschitz subspace of $C(K)$, and there exists 
$\lambda>0$ such that $L(f) \leq \lambda \|f\|$ for every $f \in X$.
\end{itemize}
\end{prop}

\noindent
\begin{proof} Observe that $X \cap \lip(K)=\bigcup_{n=1}^{\infty} \{f \in X: L(f) \leq n\}$ is a decomposition into countably many closed balanced convex sets. If $X \cap \lip(K)$ 
is not of first category in $X$, then there is $\delta>0$ such that $\delta B_{X} \subset  \{f \in X: L(f) \leq n\}$ for some $n \in {\Bbb N}$. 
By homogeneity, we have $L(f) \leq \lambda\|f\|$ with $\lambda=\delta^{-1}n$
for every $f \in X$. In particular $X \subset \lip(K)$.\end{proof}

\begin{prop}\label{lipschitz}
Let $J: X \rightarrow C(K)$ be an isomorphic embedding. Then $J(X) \subset \lip(K)$ if and only if $J^{*}|_K$ is Lipschitz from $d$ to the norm of $X^{*}$, where $J^*$ denotes the adjoint mapping from $C(K)^*$ into $X^*$. In such a case, there is $\delta>0$ such that 
$$ \delta B_{X^*}  \subset \cconv^{w^*}(J^*(K) \cup (- J^*(K) )) .$$ 
\end{prop}

\begin{proof} If $J^{*}|_K$ is Lipschitz, then any function $J(x)$ is Lipschitz as well, since 
 $J(x)(t)=J^*(t)(x)$. Reciprocally, assume that $J(X) \subset \lip(K)$.
By Proposition~\ref{easy} there is $\lambda >0$ such that $L(f) \leq \lambda \|f\|$
for every $f \in J(X)$. Now,  if $x \in B_X$ and $t_1,t_2 \in K$ then
$$ |J^{*}(t_1)(x) - J^{*}(t_2)(x)| = |J(x)(t_1)-J(x)(t_2)| \leq \lambda \, d(t_1,t_2) . $$
Taking supremum on $x \in B_X$ we get $\|J^{*}(t_1) - J^{*}(t_2)\| \leq \lambda \, d(t_1,t_2)$.\\
The formula
$$ |\!|\!| x |\!|\!| = \| J(x) \|_\infty = \sup\{ |J(x)(t)| : t \in K\} = \sup\{ |x^*(x)| : x^* \in J^*(K) \} $$
defines an equivalent norm on $X$, that is, $J^*(K)$ is a norming set. The bipolar theorem implies that the balanced $w^*$-closed convex hull of $J^*(K)$ is an equivalent norm, giving so the last statement. Moreover, some extra computations 
we could specify that $\delta = \|J^{-1}|_{J(X)}\|$.
\end{proof}

Now we will turn our attention to isometric embeddings. The set of extreme points of a convex subset $C$ is denoted by $\Ext(C)$.

\begin{lema}\label{lema}
Let $X, Y$ be Banach spaces and let $J: X \rightarrow Y$ be a linear operator. Then $J$ is an isometric embedding
if and only if
$$ \Ext(B_{X^{*}}) \subset J^*(\Ext(B_{Y^*})). $$
\end{lema}

\begin{proof}
Note that, in general, $J: X \rightarrow Y$ is an isometric embedding
if and only if $J^{*}(B_{Y^{*}})
=B_{X^{*}}$ (the less easy part relies on the Hahn--Banach theorem). 
Now, it is an easy exercise to prove that, for any $x^* \in \Ext(B_{X^*})$, any extreme point of the convex $w^*$-compact set $(J^*)^{-1}(x^*) \cap B_{Y^*}$ must be an extreme point in $B_{Y^*}$.
\end{proof}

\begin{theo}\label{iso_char}
There is an isometric embedding $J: X \rightarrow C(K)$ as a Lipschitz subspace if and only if there exists a mapping $\Psi: K \rightarrow B_{X^*}$ which is continuous for the weak$^*$ topology, Lipschitz for the metrics $d$-$\| \cdot \|$ and such that 
$$ \Ext(B_{X^*}) \subset \Psi(K) \cup (- \Psi(K)). $$
In such a case, $\Psi= J^*|_{K}$ and  $J(x)(t)=\Psi(t)(x)$.
\end{theo}

\begin{proof}
Let us call $J: X \rightarrow C(K)$ the isometric embedding. By Proposition \ref{lipschitz} we already know that $\Psi =J^*|_K$ is Lipschitz, it is obviously continuous from $K$ to the $w^*$-topology and satisfies the required condition by Lemma \ref{lema} since $\Ext(B_{C(K)^*})=K \cup (-K)$. For the other implication, consider the proposed formula $J(x)(t)=\Psi(t)(x)$. Evidently, $J$ is a linear operator with $\|J\| \leq 1$ that satisfies $J^{*}|_K=\Psi$. Lemma \ref{lema} implies that $J$ is an isometric embedding of $X$ into $C(K)$ as a Lipschitz subspace.
\end{proof}

The idea behind the following result was used by Donoghue \cite{donoghue}  for the construction of Peano-type filling curves. 

\begin{coro}\label{dona}
Let $X$ be a Gâteaux smooth Banach space and let $J: X
\rightarrow C(K)$ be an isometric embedding. Then
$$ S_{X^{*}} \subset J^{*}(K) \cup (- J^{*}(K)). $$
Moreover, if $X$ is infinite-dimensional, then
$$ B_{X^{*}} = J^{*}(K) \cup (- J^{*}(K)). $$
\end{coro}

\begin{proof} 
The idea is to use the fact that if $X$ is Gâteaux smooth, then $\Ext(B_{X^*})$ contains the set of norm-one attaining functionals 
$$N\hspace{-3pt}A = \{ x^* \in  S_{X^{*}} : \exists x \in S_X, x^*(x)=1 \}. $$
This result is a folklore to us, but we were not able to find a reference, so we will provide a full proof.
Let $x^* \in N\hspace{-3pt}A$ that attains its norm at $x$.
Shmulyan characterization of Gâteaux smoothness \cite[Corollary 7.22]{banach}, implies that for every $x \in S_X$, its norm attaining functional $x^* \in S_{X^*}$ given by the Hahn-Banach theorem is unique.
Suppose that $x^*=(y^*+z^*)/2$ with $y^*,z^* \in B_{X^*}$. Then $y^*(x), z^*(x) \leq 1$, but $y^*(x)+z^*(x)=2$, That implies $y^*(x)=z^*(x)=1$, and thus $y^*=z^*=x^*$, proving so that $x^* \in \Ext(B_{X^*})$.
The Bishop-Phelps Theorem \cite[Theorem~7.41]{banach}, says that $N\hspace{-3pt}A$ is norm dense in $S_{X^*}$, and therefore 
$\Ext(B_{X^*})$ is dense too.
Now, note that the $w^*$-compactness of $J^{*}(K)$ implies that it is norm closed and so it is $J^{*}(K) \cup (- J^{*}(K))$. The combination with the previous fact gives that 
$$S_{X^*} \subset J^{*}(K) \cup (- J^{*}(K)).$$ 
For the second statement, just observe that $S_{X^*} $ is weak$^*$-dense in $B_{X^*} $ when
$X$ is infinite-dimensional.
\end{proof}

Let us show how Corollary \ref{dona} can produce filling curves. If $K=[0,1]$ and $X$ has dimension bigger than $2$, then 
$J^*([0,1])$ and its antipodal set $-J^*([0,1])$ cover $ S_{X^{*}}$, that has dimension $2$ or bigger as a manifold. Suppose that $X={\Bbb R}^3$ with the Euclidean norm. Then $J^*([0,1])$ has nonempty interior relative to ${\Bbb S}^2$, the $2$-dimensional sphere, so by Baire's theorem, it contains a patch homeomorphic to a circle (or the square, equivalently) that is totally covered by $J^*([0,1])$. In order to get the circle filled with the image of $[0,1]$ we may take a suitable projection of $B_{{\Bbb R}^3}$ onto a plane, such that the patch goes to a circle, and then use a retraction of the plane onto the circle (the radial retraction, for instance) to finish the work.\\

However, this method to produce Peano-type filling curves does not skip the topological difficulties: they are simply hidden. The proof of the Banach-Mazur requieres these two facts: any metrizable compact is an onto continuous image of the Cantor set; and, continuos mappings from the Cantor set, identified as a subset of $[0,1]$, to ${\Bbb R}$ or, more generally, to a Banach space, can be continuously extended to the domain $[0,1]$ with values in the closed convex hull (of the former range). Both facts together produce a space filling curve in this way: put the square as a continuos image of the Cantor set and extend the mapping continuously to $[0,1]$.


\section{When $d$ metrizes $K$}

The following result is essentially a folklore, although with different variations (see \cite[Exercise 2.59]{banach}, for instance). 
For an interesting version involving vector-valued functions, along with several other beautiful applications of the Baire theorem to subspaces of C(K) made up of regular functions, we refer to \cite{Godefroy}.

\begin{theo}
If $d$ metrizes $K$, then all the Lipschitz subspaces of $C(K)$ are finite dimensional.
\end{theo}

\begin{proof}
If $X$ is a Lipschitz subspace, then $B_X$ is a bounded and complete set of functions of $C(K)$. By Proposition \ref{easy} we also know that $B_X$ is equicontinuous. Therefore, by Ascoli's theorem, $B_X$ is norm compact and so $X$ is finite dimensional.
\end{proof}

On the other hand, if $d$ does not metrize $K$, there $C(K)$ contains infinite dimensional Lipschitz subspaces.

\begin{theo}\label{cecero}
If the topology generated by $d$ is strictly finer than the topology of $K$, then $C(K)$ contains a Lipschitz subspace isometric to $c_0$.
\end{theo}

\begin{proof}[Restricted proof.]
We will assume here that $K$ is metrizable, or more generally, sequentially compact. The full proof will be provided in the next section.
Since the topology generated by $d$ cannot be compact, there is $\varepsilon>0$ and a sequence $(t_n) \subset K$ such that $d(t_n,t_m) > 6\varepsilon$ for every $n \not = m$. We may assume that the sequence is converging to some $t_0 \in K$, and removing one more element if necessary, we may assume that $d(t_n,t_0) \geq 3\varepsilon$. Suppose we pick $s_n \in B[t_n,\varepsilon]$ for infinitely many $n \in {\Bbb N}$ and let $s \in K$ a cluster point of $(s_n)$. Lower semicontinuity of $d$ implies $d(t_0,s) \leq \varepsilon$ and thus any set of the form
$$ B[t_0,\varepsilon] \cup \bigcup_{k=n}^\infty B[t_k,\varepsilon] $$
is closed for every $n \in {\Bbb N}$. Take disjoint open sets $U_1$ and $V_1$ such that $B[t_1,\varepsilon] \subset U_1$ and
$$ B[t_0,\varepsilon] \cup \bigcup_{k=2}^\infty B[t_k,\varepsilon] \subset V_1$$
Take now disjoint open sets $U_2,V_2$ with $ \overline{U_2},  \overline{V_2} \subset V_1$ and such that $B[t_2,\varepsilon] \subset U_2$ and 
$$ B[t_0,\varepsilon] \cup \bigcup_{k=3}^\infty B[t_k,\varepsilon] \subset V_2 . $$
Following in this way we will get a sequence of open sets $(U_n)$ such that $B[t_n,\varepsilon] \subset U_n$ and the sequence 
$(\overline{U_n})$ is pairwise disjoint. Applying Theorem \ref{interpol} there is a continuous function $f_n : K \rightarrow [0,1]$ such that $f(t_n)=1$, $f_n|_{K \setminus U_n}=0$ and $f_n$ is Lipschitz with constant at most $\varepsilon^{-1}$.\\
For any $(a_n) \in c_0$, the series $\sum_{n=1}^\infty a_n f_n$ is uniformly convergent on $K$ and so define a continuous function $f$. Note that $f$ is Lipschitz with constant no larger than $2\varepsilon^{-1} \| (a_n) \|_\infty$. Therefore, the mapping 
$$ (a_n) \rightarrow \sum_{n=1}^\infty a_n f_n $$
defines an isometry of $c_0$ into a Lipschitz subspace of $C(K)$.
\end{proof}

The combination of the two previous results of this section gives the following.

\begin{coro}
The lower semicontinuous metric $d$ metrizes $K$ if and only if all the Lipschitz subspaces of $C(K)$ are finite dimensional.
\end{coro}

A linear subspace $A$ of a Banach space $X$ is said to be {\it spaceable} if it contains an infinite-dimensional closed subspace of $X$. Therefore, the last corollary can be reformulated as $\lip(K) \cap C(K)$ is spaceable (in $C(K)$) if and only if $d$ does not metrizes $K$.\\

The dual point of view of isometric embeddings is useful even in finite dimension.

\begin{prop}
A finite dimensional polyhedral space $X$ embeds isometrically into $\ell_\infty^n$ if and only if $2n$ is not lesser than the number of faces of $B_X$.
\end{prop}

\begin{proof}
Let $2N$ be the number of faces of $B_X$ and let $\Ext(B_{X^*}) = \{ \pm x^*_1,\dots, \pm x^*_N \}$ and 
$K = \{1,\dots,n\}$.
Our hypothesis says that $n \geq N$. Define $\Psi: K \rightarrow \Ext(B_{X^*})$ by $\Psi(i)=x_i^*$ for $i \leq N$ and 
$\Psi(i)=x_N^*$ otherwise. Evidently, $\Ext(B_{X^*}) = \Psi(K) \cup (-\Psi(K))$, so we can apply Theorem~\ref{iso_char} to
produce an isometric embedding $J: X \rightarrow C(K)=\ell_\infty^n$ defined by $J(x)(i)=\Psi(i)(x)$.
\end{proof}

For the remaining results of this section, the hypothesis `metric compact' stresses the fact that $K$ is metrized by $d$.

\begin{prop}[\cite{jonard_raja}]\label{poli}
If $K$ is an infinite metric compact space, then $C(K)$ contains an isometric copy of  any finite-dimensional polyhedral space consisting of Lipschitz functions .
\end{prop}

\begin{proof} 
If $X$ is polyhedral and finite dimensional, its dual $X^{*}$ is also polyhedral and therefore $\Ext(B_{X^{*}})=\{x^*_1,\dots,x^*_N\}$ is a finite set. We may take different points $\{t_n\}_{n=1}^N \subset K$ and disjointly supported Lipschitz functions 
$\psi_n:K \rightarrow [0,1]$ such that $\psi_n(t_m)=0$ if $n \not = m$ and $\psi_n(t_n)=1$. 
The map defined by $\Psi(t)= \sum_{n=1}^N \psi_n(t) x^*_n$ is Lipschitz and 
$\| \Psi(t) \| \leq 1$ for every $t \in K$, as the sum is disjointly supported, and so $\Psi(K) \subset B_{X^*}$.
Since $\Ext(B_{X^*}) \subset \Psi(K)$, Theorem~\ref{iso_char} implies that the linear operator $J: X \rightarrow C(K)$ defined by $J(x)(t)=\Psi(t)(x)$ is an isometric embedding.
\end{proof}

The fact mentioned in the introduction is now explained in the following result, that implies a relation between the dimension of the Lipschitz copies of the Euclidean spaces and the dimension of $K$.

\begin{thm}[\cite{jonard_raja}]\label{main22}
Let $K$ be a metric compact space and let $n \in {\Bbb N}$. The following are equivalent:
\begin{itemize}
\item[(i)] there is an onto Lipschitz mapping $\phi: K \rightarrow {\Bbb I}^n$;
\item[(ii)] $C(K)$ contains isometric Lipschitz copies of all $(n+1)$-dimensional Banach spaces;
\item[(iii)] $C(K)$ contains an isometric Lipschitz copy of $({\Bbb R}^{n+1}, \| \cdot \|_2)$.   
\end{itemize}
\end{thm}
 
\begin{proof}[Hint of proof.]
If there is an onto  Lipschitz mapping $\phi: K \rightarrow {\Bbb I}^n$, then with the help of the stereographic projection is possible to build a Lipschitz mapping $\psi: K \rightarrow {\Bbb S}^n$, such that ${\Bbb S}^n= \psi(K) \cup (-\psi(K))$. On the other hand, if there a Lipschitz mapping $\psi: K \rightarrow {\Bbb S}^n$ such that $\psi(K)$ has nonempty interior relative to ${\Bbb S}^n$, then is possible to find an onto Lipschitz mapping $\phi: K \rightarrow {\Bbb I}^n$.
\end{proof}

We will consider the Hilbert cube $[0,1]^{\Bbb N}$ with the metric
$$ d((a_n), (b_n)) = \sum_{k=1}^\infty 2^{-k} |a_k - b_k| .$$

\begin{coro}
The space of continuous functions on the Hilbert cube $C([0,1]^{\Bbb N})$ contains Lipschitz copies of all finite dimensional Banach spaces.
\end{coro}
 
 \begin{rema}
 If there is a family of onto Lipschitz mappings $\phi_n: K \rightarrow {\Bbb I}^n$ with the Lipschitz constants uniformly bounded, then there is an onto Lipschitz mapping $\Phi: K \rightarrow [0,1]^{\Bbb N}$. Indeed, for the standard metric on ${\Bbb I}^n$, the mappings  
 $\eta_n : {\Bbb I}^n \rightarrow [0,1]^{\Bbb N}$ defined by $\eta_n ((a_k)_{k=1}^n) = (a_k)_{k=1}^\infty$ taking $a_k=0$ for $k>n$ are equi-Lipschitz. The family of mappings $(\eta_n \circ \phi_n)$ is equicontinuous, therefore there is a uniformly convergent subsequence whose limit $\Phi$ is Lipschitz and onto.
 \end{rema}


\section{Fragmentability and universality}

The following definition applies to general topological spaces, but we will restrict ourselves to the compact frame.
We say that $K$ is fragmented by the metric $d$ if for every
nonempty subset $A \subset K$ and every $\varepsilon >0$ there is $U \subset K$ open
such that $A \cap U \not = \emptyset$ and $\diam(A \cap U) < \varepsilon$, where
`$\diam$'  is the diameter measured with respect to $d$. For a metrizable compact space, fragmentability with respect to a lsc metric $d$ is the same that separability in the $d$-topology. 
As the proof is not easy to find, we will provide some ideas and the interested reader can complete the details.
If $K$ is $d$-separable, then any closed subset $A \subset K$ can be expressed as $A=\bigcup_{n=1}^\infty A_n$ with 
$\diam(A_n) <\varepsilon$ for all $n$. Then apply Baire's theorem.
Now, if $K$ were not $d$-separable, then it would contain an uncountable set $A$ with points $\varepsilon$-separated for some $\varepsilon>0$. Dentability allows us to remove the points of $A$ one by one using intersection with open sets in a transfinite inductive process that eventually will finish with the empty set. As we may use the open sets from a countable basis, that would lead to a contradiction.\\

A compact that is fragmentable some lsc metric is called a {\it Radon-Nikodym} compact since fragmentability of the weak$^*$ compact subsets characterize dual Banach spaces with the Radon-Nikodym property (taking the name of the celebrated result on differentiation of measures), see \cite{DGZ, banach} for instance. That characterization actually comes from a combination of the results two fundamental works: given a Banach space $X$, Stegall \cite{ste} proved that $X^*$ has the Radon-Nikodym property if and only if every separable subspace of $X$ has a separable dual, and those Banach spaces $X$ are usually referred as {\it Asplund spaces} after the work of Namioka and Phelps \cite{NP}, who also established the link with fragmentability there.

\begin{theo}\label{Asplund_lip}
If $K$ is fragmentable by $d$, then any Lipschitz subspace of $C(K)$ is Asplund.
\end{theo}

\begin{proof}
Let $X$ be a Lipschitz subspace of $C(K)$ and let $J: X \rightarrow C(K)$ be the isomorphic embedding. 
By Proposition \ref{lipschitz}, we know that $J^*$ is continuos from $K$ to $(X^*, w^*)$ and Lipschitz.
Since 
fragmentability of compact spaces is preserved by continuous mappings that also are Lipschitz for the metrics
\cite[Lemma 2.1]{Namioka}, we deduce that $J^*(K)$ is a 
$w^*$-compact fragmented by the norm.
Also by Proposition \ref{lipschitz}, we have
$$ \delta B_{X^*}  \subset \cconv(J^*(K) \cup (- J^*(K) )) , $$
that implies the norm fragmentability of $B_{X^*}$ since that property is preserved by finite unions (easy), $w^*$-closed convex hulls \cite[Theorem 2.5]{Namioka} and subsets (trivial).
Therefore, $X^*$ has the Radon-Nikodym property and so $X$ is Asplund.
\end{proof}

\begin{rema}
There are more properties that can be transferred from $K$ to $B_{X^*}$ through the mapping $J^*$, used in the proof of Theorem~\ref{Asplund_lip}, like the fragmentability, with consequences for $X$. For instance, assume that $K$ is a 
{\it descriptive compact} and fragmentable (by $d$, of course) and $X$ is isomorphic to a Lipschitz subspace of $C(K)$. Since descriptiveness is preserved by continuous images \cite[Corollary 3.4]{OR} and $w^*$-closed convex hulls \cite[Corollary 3.7]{OR}, we deduce that $B_{X^*}$ is descriptive and norm fragmentable. Both properties together imply that $X^*$ admits an equivalent locally uniformly rotund dual norm \cite[Theorem~1.3]{OR}. Descriptive topological spaces were introduced by Hansell \cite{Hansell}.

\end{rema}

\begin{lema}\label{lema_cantor}
Assume $K$ is not fragmentable by $d$. Then there exits $\varepsilon >0$ and two families $(U_s)$ and $(V_s)$ indexed by 
$\{0,1\}^{<{\Bbb N}}$ satisfying:
\begin{enumerate}
\item $U_{s\frown 0} \cup U_{s\frown 1} \subset U_s$ and $V_{s\frown 0} \cup V_{s\frown 1} \subset V_s$ for every $s$;
\item $\overline{V_s} \subset U_s$ for every $s$;
\item $d(\overline{V_s}, K \setminus U_s) > \varepsilon$ for every $s$;
\item $U_{s\frown 0} \cap U_{s\frown 1} = \emptyset$ for every $s$.
\end{enumerate}
\end{lema}

\begin{proof}
If $K$ is not fragmentable by $d$ there is a closed subset $A \subset K$ and $\varepsilon>0$ such that every nonempty relatively open set of $A$ has diameter greater than $3\varepsilon$. 
The construction of the families will be done by induction on the length of the sequence $s$, adding one more condition: points $(x_s) \subset A$ with $x_s \in V_s$ and $d(x_{s \frown 0}, x_{s \frown 1}) > 3\varepsilon$.
Take two points $x_0, x_1 \in A$ with $d(x_0,x_1) > 3\varepsilon$. Using the lower semicontinuity of the metric, take now two open sets $V_0,V_1$ with $x_0 \in V_0$ and $x_1 \in V_1$ such that the distance between 
$\overline{V_0}$ and $\overline{V_1}$ is at least $3\varepsilon$. The closed sets 
$ B[\overline{V_0}, \varepsilon] $ and $ B[\overline{V_1}, \varepsilon]$
are disjoint. Finally take disjoint open sets $U_0 \supset  B[ \overline{V_0}, \varepsilon ]$ and
$U_1 \supset  B[\overline{V_1}, \varepsilon ]$.\\
Assume everything is built for $|s| \leq n$. For a given $s$ with $|s|=n$ we will construct the objects for $s \frown 0$ and $s \frown 1$.
Since $x_s \in A \cap V_s$, we have $A \cap V_s \not = \emptyset$. 
This relatively open set of $A$ has diameter greater than $3\varepsilon$. Take points $x_{s \frown 0}, x_{s \frown 1} \in A \cap V_s$ with $d(x_{s \frown 0}, x_{s \frown 1}) >3\varepsilon$. 
Take open sets $V_{s \frown 0},V_{s \frown 1} \subset V_s$ with $x_{s \frown 0} \in V_{s \frown 0}$ and $x_{s \frown 1} \in V_{s \frown 1}$ such that the distance between 
$\overline{V_{s \frown 0}}$ and $\overline{V_{0s \frown 1}}$ is at least $3\varepsilon$. The closed sets 
$ B[ \overline{V_{s \frown 0}}, \varepsilon ]$ and $ B[\overline{V_{s \frown 1}}, \varepsilon] $
are disjoint, so they can be separated by open sets $U_{s \frown 0}$ and $U_{s \frown 1}$. Without loss of generality we may assume $U_{s \frown 0}, U_{s \frown 1} \subset U_s$. That completes the induction argument.
\end{proof}

\begin{rema}\label{rema_metriza}
If $K$ is besides metrizable, say by a metric $\rho$, then we may add to the construction the condition that $\rho$-diameter of $\overline{U_s}$ goes to $0$ with $|s| \rightarrow \infty$. Indeed, the sets $U_s$ are contained into sets $V_s$ that are taken to be neighbourhoods of the given points $x_s$. Therefore, we may take $V_s$ contained in the ball of center $x_s$ and radius  
$|s|^{-1}$ (with respect to $\rho$), for instance.
\end{rema}

The following is the second key result of Lipschitz subspaces outside the metrizable case.

\begin{theo}[\cite{raja1}]\label{main}
The following statements are equivalent:
\begin{itemize}
\item[(i)] $K$ is not fragmentable by $d$;
\item[(ii)] $C(K)$ contains an isometric Lipschitz copy of $\ell_1$;
\item[(iii)] $C(K)$ contains an isomorphic Lipschitz copy of $\ell_1$.
\end{itemize}
\end{theo}

\begin{proof}
$(i) \Rightarrow (ii)$ 
If $K$ is not $d$-fragmentable, we may produce a Cantor-like closed subset using Lemma \ref{lema_cantor} this way
$$ H = \bigcap_{n \in {\Bbb N}} \bigcup_{|s|=n} \overline{V_s}. $$
For every $t \in H$ there is a unique $\sigma(t) \in \{0,1\}^{\Bbb N}$ such that $t \in \bigcap_{n \in {\Bbb N}}  \overline{V_{\sigma(t) |n}}$. The mapping $ \Sigma: H \rightarrow \{-1,1\}^{\Bbb N} $
defined by taking $\Sigma(t)$ the sequence $\sigma(t)$ after changing the $0$'s by $-1$'s. It is easy to check that $\Sigma$ is onto and continuous. Moreover, if $d(t_1,t_2) \leq 2\varepsilon$, then $\Sigma(t_1)=\Sigma(t_2)$.
Let $p_n$ the projection on the n'th coordinate of
$\{-1,1\}^{\Bbb N}$ and consider the function $p_n \circ \Sigma$ and note that it is continuous and $\varepsilon^{-1}$-Lipschitz.
By Theorem \ref{interpol}, there is a continuous extension $f_n$ of $p_n
\circ \Sigma$ to $K$ with the same Lipschitz bound $\varepsilon^{-1}$.
The sequence $(f_n)$ is equivalent to the canonical basis of
$\ell_1$. Indeed, given numbers real numbers $(a_n)$ for $i=1,
\dots, m$ there is $x \in H$ such that $f_n(x)=\mbox{sign}(a_n)$ and thus
$$ \| \sum_{n=1}^{m} a_n f_n \| = \sum_{n=1}^{m} |a_n| $$
which means that $E=\overline{span}^{\|.\| }\{f_n: n \in {\Bbb N}\}$ is isometric to $\ell_1$.
An easy computation shows that if $f \in E$, then $f$ is $\varepsilon^{-1} \|f\|$-Lipschitz.\\
$(ii) \Rightarrow (iii)$ It is obvious.\\
$(iii) \Rightarrow (i)$  If $C(K)$ contains an isorphic Lipschitz copy of $\ell_1$, then Theorem \ref{Asplund_lip} implies that $K$ cannot be fragmentable by $d$ because $\ell_1$ is not Asplund.
\end{proof}

\begin{rema}
It is possible to add an equivalent condition to Theorem \ref{main}: there is an $d$-equicontinuous bounded sequence 
$(f_n) \subset C(K)$ equivalent to the canonical basis of $\ell_1$, see \cite{raja1}.
\end{rema}

One longstanding problem in Banach theory was to know if a separable space $X$ having a non separable dual must contain a copy of 
$\ell_1$. A nonseparable reformulation is whether Asplundness equals not containing $\ell_1$. The problem was solved negatively by James \cite{James} and Lindenstrauss and Stegall \cite{LS}, independently. However, we have the following {\it boutade}.

\begin{coro}[\cite{raja1}]\label{coro_Asplund}
A Banach space $X$ is Asplund if and only if $C(B_{X^*})$ does not contain a Lipschitz copy of $\ell_1$.
\end{coro}

Endowing a compact with the discrete metric we retrieve this classic result.

\begin{coro}
$K$ is not scattered if and only if $C(K)$ contains an isomorphic copy of $\ell_1$.
\end{coro}

\begin{proof}[Full proof of Theorem \ref{cecero}]
If $K$ is fragmentable by $d$ then $K$ is sequentially compact \cite{Namioka}, therefore the restricted proof above gives the result. Otherwise, if $K$ is not fragmentable by $d$ we may use the sets built in Lemma \ref{lema_cantor} in this way. Let $s_n$ be the sequence made up of  $n-1$ zeroes followed by a single $1$.
The sets $(U_{s_n})$ are disjoint. Indeed, fix $m<n$ and note that $U_{s_m}$ is disjoint with $U_{(0,\dots,0)}$ ($m$ zeroes), that contains $U_{s_n}$. Using Theorem \ref{interpol}, there are continuous $\varepsilon^{-1}$-Lipschitz functions $f_n: K \rightarrow [0,1]$ such that $f_n|_{V_{s_n}} = 1$ and  $f_n|_{K \setminus U_{s_n}} = 0$. Proceeding like in the restricted proof, those functions generate a Lipschitz subspace of $C(K)$ isometric to $c_0$.
\end{proof}

Now we turn our attention to the ``Lipschitz universality'' of $C(K)$ spaces. 
We say that a space $C(K)$ is Lipschitz universal if $C(K)$ contains an isometric Lipschitz copy of $X$ for any separable Banach space $X$.
Let $\Delta$ be the Cantor space $\{0,1\}^{\Bbb N}$ together with the discrete metric.

\begin{prop}
The space $C(\Delta)$ is isometric Lipschitz universal for the separable Banach spaces. If $K$ is a metrizable compact together a lsc metric such that $C(K)$ is isomorphic Lipschitz universal for the separable Banach spaces, then $K$ contains a subset equivalent to $\Delta$, that is, homeomorphic and Lipschitz isomorphic. 
\end{prop}

\begin{proof}
By Mazur's theorem, $C(\Delta)$ is isometric universal in the standard sense and every $f \in C(\Delta)$ is Lipschitz with respect to the discrete norm. On the other hand, if $C(K)$ is isomorphic Lipschitz universal for the separable Banach spaces, then $K$ is not fragmentable by the associated metric $d$. The proof of Theorem \ref{main} together Remark \ref{rema_metriza} provide a set $H$ that is homeomorphic to the Cantor space and its points are uniformly separated, with separation bounded below by $2\varepsilon$. In order the restriction of $d$ to $H$ be Lipschitz isomorphic to the discrete metric is enough that $d$ be bounded. That is not ensured by the hypotheses, so we may proceed this way. Take any $t_0 \in H$ and consider the closed balls $B[t_0, n]$ for $n \in {\Bbb N}$, that are closed with respect to the topology of $H$ too. Since  $H \subset \bigcup_{n=1}^\infty B[t_0, n]$, there is one ball with nonempty interior (with respet to $H$). Now, any nonempty open set of the Cantor space contains an homeomorphic copy of the Cantor space itself. Indeed, the standard basis for the product topology is composed of sets of the form
$$ (a_1,a_2,\dots,a_n) \times \{-1,1\}^{\Bbb N} $$
that are homeomorphic to the Cantor space itself.  That provides us with a copy of the Cantor space where $d$ is bounded.
\end{proof}

We do not know if any metrizable $K$ not fragmentable with respect to $d$ is Lipschitz universal for the separable Banach spaces. Actually, we do not know the answer for $K=[0,1]^{\Bbb N}$ endowed with the supremum norm (that is, essentially, $B_{\ell_\infty}$ with the weak$^*$ topology and the norm metric). A main issue here is that the method to build linear extension operators within spaces of continuous functions seldom preserve Lipchitzness, see \cite{BL} or \cite{Wojta} for instance.


\section{An ordering for compacta with lsc metrics}

We have seen that the problem of identifying Lipschitz subspaces reduces to the study of applications between compact spaces that are Lipschitz with respect to the associated lsc metrics. In this section we will restrict our attention to metrizable compact spaces with associated bounded lsc metrics. Define an order among this class by $ K_1 \preceq K_2$ if there exists an onto continuous mapping $\phi: K_2 \rightarrow K_1$ that is also Lipschitz for the metrics. The definition implies trivially the following observation.

\begin{prop}
If $ K_1 \preceq K_2$, any isometric (resp. isomorphic) Lipschitz subspace of $C(K_1)$ is an isometric (resp. isomorphic) Lipschitz subspace of $C(K_2)$.
\end{prop}

In the order $\preceq$ the singleton space plays the role of minimum. On the other hand, 
$\Delta$ is a  maximum, however it is not unique. For instance, $\Delta \uplus [0,1]$ (disjoint topological union) is a maximum too. As to intermediate elements in the order $\preceq$, the most interesting example is provided by the Mazur mapping between the unit balls of Lebesgue sequence spaces, namely  
$\Phi_{q_1,q_2}: B_{\ell_{q_1}} \rightarrow B_{\ell_{q_2}}$ defined by
$$ \Phi_{q_1,q_2}(\, (x_n)_{n \in {\Bbb N}} \,) := (\mbox{sign}(x_n)|x_n|^{q_1/q_2})_{n \in {\Bbb N}} $$
which is  Lipschitz for $1 \leq q_2 \leq q_1 < \infty$, see the proof of \cite[Theorem~12.50]{banach}. The mapping is obviously continuous for the pointwise topologies, that make the balls compact (actually, they are dual unit balls with the weak$^*$ topology). Therefore $B_{\ell_{q_2}} \preceq B_{\ell_{q_1}}$ whenever $q_2 \leq q_1$.

\begin{prop}
Suppose that $C(K)$ contains an isometric Lipschitz copy of some $\ell_p$ with $p \in (1,+\infty)$, then $C(K)$ contains an isometric Lipschitz copy of $\ell_{p'}$ for every $p' \in [p,+\infty)$.
\end{prop}

\begin{proof}
Let $q, q' \in (1,+\infty)$ be the conjugate exponents of $p, p'$ respectively. Let $\Psi: K \rightarrow B_{\ell_q}$ witnessing the isometric embedding of $\ell_p$ as Lipschitz  subspace of $C(K)$. Since $\ell_p$ is smooth, by Corollary \ref{dona} we have
$B_{\ell_q}= \Psi(K) \cup (-\Psi(K))$. Note that the Mazur mapping satisfies $\Phi_{q,q'}(-x)= -\Phi_{q,q'}(x)$. Therefore, 
$$  B_{\ell_q'}= \Phi_{q,q'} (B_{\ell_q} )= \Phi_{q,q'}(\Psi(K)) \cup \Phi_{q,q'} ((-\Psi(K))) $$
$$ =  (\Phi_{q,q'} \circ \Psi )(K)
\cup  (-  (\Phi_{q,q'} \circ \Psi )(K)), $$
which implies, by Theorem \ref{iso_char}. that the isometric embedding of $\ell_{p'}$ as a Lipschitz subspace of $C(K)$.
\end{proof}

Now we will show the use of the Szlenk index as an obstacle for Lipschitz embeddings 
For any closed subset $A \subset K$ 
we define a set derivation
$$ \langle A \rangle'_{\varepsilon} = \{x \in A: \forall \, U \, \mbox{~neighbourhood of~} x,
 \diam(A \cap U) \geq \varepsilon \} ,$$
where the diameter is computed with respect to $d$.
By iteration, the sets $\langle A \rangle^{\gamma}_{\varepsilon}$ are defined
for any ordinal $\gamma$, taking intersection in the case of limit ordinals.
The Szlenk indices of $K$ with respect to $d$ are ordinal numbers defined by
$$ Sz(K, \varepsilon) = \inf \{ \gamma : \langle K \rangle^{\gamma}_{\varepsilon}
 =\emptyset \}$$
and  $Sz(K)=\sup_{\varepsilon>0} Sz(K, \varepsilon)$. If $K$ is fragmentable by $d$, the Szlenk indices
always exist. Otherwise, for some $\varepsilon>0$ there is an ordinal $\gamma$ such that
$\langle K \rangle^{\gamma}_{\varepsilon}=\langle K \rangle^{\gamma+1}_{\varepsilon} \not = \emptyset$.
In that case we put $Sz(K,\varepsilon)=\infty$ and $Sz(K)=\infty$ with the agreement that any ordinal number is less
than $\infty$. Let us to point out that In the frame of Banach spaces, the ``Szlenk index of a Banach space $X$'' refers to the index of $B_{X^*}$ computed as above, see \cite{Lancien2} for more information on that topic.

\begin{prop}[\cite{raja2}]\label{KKK}
If $K_1 \preceq K_2$, then there is $c>0$ such that, for all $\varepsilon>0$, we have 
$$Sz(K_1, \varepsilon) \leq Sz(K_2, \varepsilon / c).$$ 
\end{prop}

\begin{proof}
Let $\phi: K_2 \rightarrow K_1$ be the continuous surjection with Lipschitz constant $\lambda>0$. Take $c= 2\lambda$.
Since $\varepsilon>0$ is arbitrary, the statement is equivalent to $Sz(K_1, c \, \varepsilon) \leq Sz(K_2, \varepsilon)$. To prove this, it is enough to show that
$$ \langle \phi(A) \rangle'_{c \,\varepsilon} \subset \phi(\langle A \rangle'_{\varepsilon}) $$
for every closed subset $A \subset K_2$. 
Indeed, if
$x \in \phi(A) \setminus \phi(\langle A \rangle'_{\varepsilon})$,
then $\phi^{-1}(x)$ is compact subset of $A$ disjoint with $ \langle A \rangle'_{\varepsilon}$. The set $\phi^{-1}(x)$
can be covered with finitely many
open sets $U_1,\dots, U_n$ such that $\diam(A \cap U) < \varepsilon$.
Let $U=\bigcup_{k=1}^n U_k$. Note that  $U \cap \langle A \rangle'_{\varepsilon} = \emptyset$ and for every $y \in A \cap U$ then $d_1(\phi(y),x) < \lambda \, \varepsilon$, implying  $\diam(\phi(A \cap U) \leq c \, \varepsilon$.
Taking  the open set $V = K_1 \setminus \phi(A\setminus  U)$ we have
$ x \in \phi(A) \cap V \subset \phi(A \cap U) $
which implies that $x \not \in \langle \phi(A) \rangle'_{c \, \varepsilon}$.
\end{proof}

\begin{coro}[\cite{raja3}]
If $X$ is Gâteaux smooth and embeds isometrically as a Lipschitz subspace of $C(K)$, then there is $c>0$ such that, for all $\varepsilon>0$, we have 
$$ Sz(B_{X^{*}},\varepsilon) \leq  Sz(K, \varepsilon/c). $$
\end{coro}

It is possible to prove that $Sz(B_{\ell_q}, \varepsilon) \sim \varepsilon^{-q} $ for $q \in [1,+\infty)$, see \cite{raja2} for the details. Therefore $B_{\ell_{q_2}} \not \preceq B_{\ell_{q_1}}$ if $q_1 < q_2$.

\begin{coro}[\cite{raja3}]\label{elep}
Let $p \in (1,+\infty)$ and let  $q$ be its conjugate exponent. Then for every $p' \in [1,p)$, the space $\ell_{p'}$ does not embed isometrically as a Lipschitz subspace of $C(B_{\ell_q})$.
\end{coro}

Our results with Lipschitz isomorphic embeddings are not so satisfactory distinguishing among the balls $B_{\ell_q}$. The necessity of taking the closed convex hull implies some loss of information (the Szlenk index is increased by a factor of the form $\varepsilon^{-1}$, \cite[Corollary 3.7]{raja2}). However, some interesting results can be established for infinite Szlenk indices.

\begin{prop}
Let $K$ be metrizable compact together a lsc metric $d$. If the space $C(K)$ is isomorphic Lipschitz universal for the separable reflexive Banach spaces, then $K$ is not fragmentable by $d$.
\end{prop}

\begin{proof}
If $K$ was fragmentable, then $Sz(K)$ would be a countable ordinal and $Sz(B_{X^*}) \leq Sz(K)$ for all $X$ an isomorphic Lipschitz subspace of $C(K)$. Szlenk proved in \cite{szlenk} (see also \cite{Lancien2}) that there are reflexive Banach spaces $X$ with $Sz(B_{X^*})$ an arbitrarily high countable ordinal, leading to a contradiction.
\end{proof}

\section*{Acknowledgements}

In 2003, I was doing a postdoc at the Hebrew University of Jerusalem under the supervision of Joram Lindenstrauss when I obtained Corollary \ref{coro_Asplund}. I often enjoy myself remembering the morning as I came to Joram's office and told him ``Do you know that Asplundness can be characterized by the lack of copies of $\ell_1$?''. Some years later, during a Winter School in Czech Republic, I rediscovered Donoghue's surprising result and I tried to publish it. I am still indebted to Bill Johnson who kindly informed me that I was 50 years late in his nonacceptance letter. Finally, I would also thank the referees for their many corrections and comments that have allowed me to greatly improve the readability and appearance of the paper.

\vspace{1cm}

\begin{flushright}
Departamento de Matem\'aticas\\ Universidad de Murcia\\
Campus de Espinardo\\ 30100 Espinardo, Murcia, SPAIN\\
E-mail: matias@um.es
\end{flushright}


{\footnotesize

\begin{thebibliography}{99}

\bibitem{linea}{\sc R. M. Aron, L. Bernal-González, D. Pellegrino, J. B. Seoane-Sepúlveda},
{\it  Lineability: the search for linearity in mathematics}, 
Monographs and Research Notes in Mathematics. CRC Press, Boca Raton,
FL, (2016)

\bibitem{Benya}{\sc Y. Benyamini}, 
Simultaneously continuous retractions on the unit ball of a Banach space,
{\it Journal of Approximation Theory} 38, 28--42 (1983)


\bibitem{BL}{\sc Y. Benyamini, J. Lindenstrauss}, 
{\it ~Geometric Nonlinear Functional Analysis.
Vol.~1}, American Mathematical Society Colloquium Publications 48, 2000.


\bibitem{DGZ}{\sc R. Deville, G. Godefroy, V. Zizler}
{\it Smoothness and renormings in Banach spaces},
Pitman Monographs and Surveys in Pure and Applied Mathematics, 64. Longman Scientific \& Technical,
Harlow, 1993.


\bibitem{donoghue} {\sc W. F. Donoghue},
{Continuous function spaces isometric to Hilbert space},
{\it Proc. Amer. Math. Soc.}  8  (1957),  1--2.


\bibitem{banach} {\sc M. Fabian, P. Habala, P. H\'ajek, V. Montesinos and V. Zizler},
{\it Banach Space Theory. The Basis for Linear and Nonlinear Analysis},
CMS Books in Mathematics,
Springer, New York, 2011.

\bibitem{Godefroy}{\sc G. Godefroy},
{\it Introduction aux méthodes de Baire},
Calvage $\&$ Mounet, Paris 2022.

\bibitem{Hansell}{\sc R. W. Hansell}, 
Descriptive sets and the topology of nonseparable Banach spaces, 
{\it Serdica Math}. J. 27 (2001), 1--66.

\bibitem{James}{\sc R.C. James}{~ A separable somewhat reflexive Banach space with nonseparable dual},
{\it Bull. Amer. Math. Soc.} {80} (1974), 738--743.

\bibitem{JNR}{\sc J.E. Jayne, I. Namioka, C.A. Rogers},
{Norm fragmented weak$^{*}$ compact sets},
{\it Collect. Math.} {41} (1990), 161--188.

\bibitem{jonard_raja}{\sc N. Jonard Perez, M. Raja},
{Lipschitz subspaces of C(K)},
{\it Topology Appl.} 204 (2016), 149--156.

\bibitem{JL}{\sc W. B. Johnson, J. Lindenstrauss},
{Basic concepts in the geometry of Banach spaces},
{\it Handbook of the Geometry of Banach spaces}
Vol. 1, W.B. Johnson and J. Lindenstrauss editors,
Elsevier, Amsterdam (2001), 1-- 84.

\bibitem{Lancien2}{\sc G. Lancien},
{A survey on the Szlenk index and some of its applications},
{\it RACSAM Rev. R. Acad. Cienc. Exactas Fís. Nat. Ser. A Mat.} {100} (2006), no. 1-2, 209--235.

\bibitem{LS}{\sc J. Lindenstrauss, C. Stegall}{~Examples of separable spaces which do not contain
$\ell_{1}$ and whose duals are non-separable}, {\it Studia Math.} {54} (1975), 81--105.

\bibitem{Eva}{\sc E.  Matouskova},
Extensions of continuous and Lipschitz functions, 
{\it Canad. Math. Bull.} 43 (2000), 208--217.

\bibitem{Namioka}{\sc I. Namioka}, 
{Radon-Nikod\'ym compact spaces and
fragmentability}, {\it Mathematika} {34} (1989), 258--281.

\bibitem{NP}{\sc I. Namioka, R.R. Phelps},
Banach spaces which are Asplund spaces, 
{\it Duke Math}. J. 42 (4) (1975), 735--750.

\bibitem{OR}{\sc L. Oncina, M. Raja}, 
Descriptive compact spaces and renorming, 
{\it Studia Math}. 165 (2004), no. 1, 39--52. 

\bibitem{raja1}{\sc M. Raja},
{Embedding $\ell_1$ as Lipschitz functions},
{\it Proc. Amer. Math. Soc.} 133 (2005), no. 8, 2395--2400.

\bibitem{raja2}{\sc M. Raja},
{Compact spaces of Szlenk index $\omega$}, 
{\it J. Math. Anal. Appl.} 391 (2012), no. 2, 496--509.

\bibitem{raja3}{\sc M. Raja},
{Two applications of smoothness to $C(K)$ spaces},
{\it Studia Math.} 225 (2014), no. 1, 1--7.

\bibitem{ste}{\sc C. Stegall}, 
The Radon-Nikodym property in conjugate Banach spaces, 
{\it Trans. Amer. Math. Soc.} , 206 (1975), 213--223.

\bibitem{szlenk}{\sc W. Szlenk}, The non-existence of a separable reflexive space universal for all separable reflexive
Banach spaces, {\it Studia Math}. 30 (1968) 53--61.


\bibitem{Wea}{\sc N. Weaver},
{\it Lipschitz Algebras}, 2$^{nd}$ edition, 
World Scientific 2018.

\bibitem{Wojta}{\sc P. Wojtaszczyk},
{\it Banach spaces for analysts}, 
Cambridge studies in mathematics 25,
Cambridge University Press, 1991.

\end{thebibliography}

}
\end{document}